\documentclass[12pt]{article}
\usepackage{amsmath}
\usepackage{amssymb}
\usepackage{amsthm}

\title{Two-sided radial SLE\\and length-biased chordal SLE}
\author{Laurence S. Field}
\date{April 18, 2015}

\newtheorem{thm}{Theorem}
\newtheorem{lem}[thm]{Lemma}
\newtheorem{prop}[thm]{Proposition}
\newtheorem{cor}[thm]{Corollary}
\theoremstyle{definition}
\newtheorem{defn}[thm]{Definition}
\theoremstyle{remark}
\newtheorem{rmk}[thm]{Remark}

\newcommand\D{\mathbb D}
\newcommand\Half{\mathbb H}
\newcommand\R{\mathbb R}
\newcommand\C{\mathbb C}

\newcommand\E{\mathbf E}
\newcommand\F{\mathcal F}

\newcommand\Prob{{\mathbf P}}

\newcommand\K{{\mathcal K}}
\newcommand\B{{\mathcal B}}
\newcommand\M{{\mathcal M}}
\newcommand\dc{{\rm d}}
\newcommand\dd{{\tilde d}}

\newcommand\tmass\Psi
\newcommand\lm\Lambda

\DeclareMathOperator{\crad}{crad}

\renewcommand\Im{\operatorname{Im}}
\newcommand\p\partial
\newcommand\inv{^{-1}}
\newcommand\ol\overline
\newcommand\sm\setminus
\newcommand\exc{\mathcal E}

\newcommand\SLEk{{SLE}$_\kappa$}
\renewcommand\Im{\operatorname{Im}}

\DeclareMathOperator\Area{Area}
\DeclareMathOperator\Cont{Cont}
\DeclareMathOperator\diam{diam}
\DeclareMathOperator\dist{dist}
\DeclareMathOperator\osc{osc}
\DeclareMathOperator\vol{vol}

\begin{document}

\maketitle

\begin{abstract}
We show that, for $\kappa\le 4$, the integral of the two-sided radial \SLEk\ measures over all interior points is chordal \SLEk\ biased by the path's natural length, which is its $(1+\frac\kappa8)$-dimensional Minkowski content.
\end{abstract}

\section{Introduction}

The present paper establishes a relationship between chordal SLE paths and a certain aggregate measure of two-sided radial SLE paths in simply connected planar domains. To explain the motivation and context for our results, we will first mention Brownian measures on curves.

The study of loop measures has flourished since it became apparent that Brownian loop measures and loop soups play a major role in the theory of conformal invariance~\cite{LeJan,Werner,LW}.
Loop measures can be defined algebraically in the setting of Markov chains, and at this level the relationship between loop-erased random walk, uniform spanning trees, loop measures and Gaussian free fields is already apparent~\cite{LP}.
However, in the continuous plane, the Brownian loop measure gains extra significance from its conformal invariance property.
Thanks to Lawler and Trujillo Ferreras' proof that the Brownian loop measure is the scaling limit of random walk loop measure~\cite{LTF}, 
the discrete and continuum loop measures can profitably be used together, as seen in the identification of the scaling limit of the Green's function for loop-erased random walk~\cite{BJVL}.

In the proof that the Brownian loop measure is conformally invariant~\cite{LW}, it is very important that the loops are parametrized \emph{naturally}.
In the case of Brownian motion, this means that the quadratic variation of the path up to any time $t$ is exactly $t$.
The general principle for the construction of unrooted loop measures, seen both discretely and in the continuum~\cite{LTF,LW}, is that if we take the integral of rooted measures over all possible roots, we have overcounted the true unrooted measure by a factor of the random curve's length.
(Heuristically, we want to count each curve only once, not once for each possible root.)

Motivated by the relationship seen among Brownian measures, we dedicate this paper to the proof of Theorem~\ref{thm:lengthbias}, which describes chordal \SLEk\ as an unrooting of rooted measures in a sense that we shall now describe.

In the case of chordal \SLEk, the measure rooted at an interior point~$\zeta$ is \emph{two-sided radial \SLEk\ through $\zeta$}, to be denoted $\mu_\zeta$, which should be thought of as chordal \SLEk\ restricted to those curves passing through the interior point~$\zeta$ (and normalized appropriately).\label{discuss-2sr}

The question that arises from, among other inspirations, the theory of Brownian measures on curves, is the following: how does the aggregate of rooted measures compare to the original unrooted measure?

In the underlying discrete models, it is easy to conceive of the rooted measure as the unrooted measure, \emph{restricted} (not conditioned) to those curves that pass through $\zeta$ (so that its total mass is the probability of passing through $\zeta$).
Then, in the aggregate (sum) of the rooted measures over all possible roots, every curve is counted precisely as many times as possible points it could be rooted at, which is the (discrete, graph) length of the curve. In other words, the aggregate of rooted measures is the unrooted measure biased by path length.

With Theorem~\ref{thm:lengthbias}, we prove that this result carries over to the continuum for \SLEk, if~$\kappa\le4$.
The aggregate of rooted measures is defined to be
\[
\nu=\int_D \mu_\zeta\,dA(\zeta),
\]
which can be understood as a Riemann integral under a natural metric on measures on curves.
Then, if $\mu$ is ordinary chordal \SLEk\ and $D$ is a bounded domain with analytic boundary, $\nu$ exists and
\[
\frac{d\nu}{d\mu}(\gamma)=|\gamma|,
\]
where $|\gamma|$ denotes the $(1+\kappa/8)$-dimensional Minkowski content of $\gamma$, which is the \emph{natural parametrization} for \SLEk\ (as first introduced in~\cite{LS} and more fully analyzed in~\cite{minkowski}).
In other words, the aggregate of rooted measures is the unrooted measure (chordal \SLEk) biased by the \SLEk\ curve's length under the natural parametrization.

\section{Statement of the result}

Let $0<\kappa\le 4$, and consider the chordal \SLEk\ curves defined as follows.
Let $U_t$ be a standard Brownian motion and consider the solution $g_t$ to the half-plane Loewner equation
\[
\p_t g_t(z)=\frac a{g_t(z)-U_t},\qquad
g_0(z)=0,\qquad
\text{where }z\in\ol\Half\text{ and }a=\frac2\kappa.
\]
Then, as first proved by Rohde and Schramm, there is a random simple curve $\gamma:(0,\infty)\to\Half$ such that $\gamma(0+)=0$ and the domain of $g_t$ is $\Half\sm\gamma_t$, where $\gamma_t=\gamma[0,t]$.
This curve is called the chordal Schramm--Loewner evolution with parameter $\kappa$ (\SLEk) from $0$ to $\infty$ in $\Half$.
If $D$ is a simply connected domain and $z,w$ accessible boundary points (interpreted as prime ends if necessary), we define chordal \SLEk\ from $z$ to $w$ in $D$ to be the image of chordal \SLEk\ from $0$ to $\infty$ in $\Half$ under a conformal transformation sending $\Half,0,\infty$ to $D,z,w$; we shall ignore the curve's parametrization for now.

Chordal \SLEk\ then has the following important properties:
\begin{enumerate}
\item
\emph{Conformal Invariance}: the image of chordal \SLEk\ in $D$ from $z$ to $w$ under a conformal map $f$ is chordal \SLEk\ in $f(D)$ from $f(z)$ to $f(w)$.
\item
\emph{Domain Markov Property}: given the initial segment $\gamma_t$ of chordal \SLEk\ in $D$ from $z$ to $w$, the remainder of the curve follows the law of chordal \SLEk\ in $D_t=D\sm\gamma_t$ from  $\gamma(t)$ to $w$.
\item
\emph{Reversibility}: as first proved by Zhan, for $\kappa\le 4$ the reversal of chordal \SLEk\ from $z$ to $w$ in $D$ is chordal \SLEk\ from $w$ to $z$ in $D$, after reparametrization.
\end{enumerate}

By a \emph{configuration} we shall mean a simply connected domain $D$ together with starting and ending points $z,w\in\p D$, which will be suppressed from the notation.

We want to study ``pinned'' or ``rooted'' measures for \SLEk. As mentioned in the introduction, in the case of chordal \SLEk, the measure pinned at an interior point $\zeta$ is \emph{two-sided radial \SLEk\ through $\zeta$},
which is (as far as the probability measure is concerned) just chordal \SLEk\ conditioned to pass through~$\zeta$.
Its construction is much like the construction of the Brownian bridge as a Doob $h$-process.
To understand this construction, we first need to consider the Green's function for chordal \SLEk\ at $\zeta$, which is the normalized probability that the curve passes near $\zeta$.

Let $\gamma$ be chordal \SLEk\ in some configuration~$D$, and let $\zeta\in D$.
The (Euclidean) Green's function for chordal \SLEk\ is
\[
G_D(\zeta)=\lim_{\epsilon\to0+} 
\epsilon^{d-2}\,\Prob\{\dist(\zeta,\gamma)<\epsilon\},
\]
where $d=1+\frac\kappa8$ is the Hausdorff dimension of the curve.
This should be thought of as the normalized probability that the curve passes through~$\zeta$.
Lawler and Rezaei proved that the limit exists, and in the half-plane it takes the form
\begin{equation}\label{eq:ghalf}
G_\Half(\zeta)=c_\kappa\,(\Im \zeta)^{d-2}\,(\sin\arg \zeta)^{4a-1}.
\end{equation}
Moreover, the Green's function is conformally covariant,
\[
G_D(\zeta)=|f'(\zeta)|^{2-d}\,G_{f(D)}(f(\zeta)),
\qquad\text{$f$ conformal on $D$}.
\]

As the definition suggests, if we stop a chordal \SLEk\ curve~$\gamma$ at time $\tau_n=\inf\{t:|\gamma(t)-\zeta|\le e^{-n}\}$, then $M_t^n:=G_{D_{t\wedge\tau_n}}(\zeta)$ is a martingale. 
Tilting the chordal \SLEk\ measure~$\mu$ by this martingale in the sense of Girsanov's theorem,
so that 
\begin{equation}\label{2srdefn}
\frac{d\mu_\zeta^n}{d\mu}(\gamma_t)=G_{D_{t\wedge\tau_n}}(\zeta),
\end{equation}
we obtain a finite measure $\mu_\zeta^n$ of total mass $G_D(\zeta)$ on curves $(\gamma(t):t\le\tau_n)$ starting at~$p$.
Moreover, these tilted measures are consistent: if $m<n$, then $(\gamma(t):t\le\tau_m)$ has the same law under $\mu_\zeta^n$ as under $\mu_\zeta^m$.
Therefore, there is a limit measure $\mu_\zeta$ called \emph{two-sided radial \SLEk\ through $\zeta$} on curves $\gamma:[0,\tau)\to\ol D$ that accumulate to $\zeta$ as $t\to\tau-$, where $\tau=\sup_n\tau_n<\infty$.\label{2sr-def}
Lawler proved endpoint continuity of~$\mu_\zeta$ in~\cite{Lcont}: with probability one $\gamma(t)\to\zeta$ as $t\to\tau-$.
Given a curve $\gamma:[0,\tau]\to\ol D$ sampled from $\mu_\zeta$ (so that $\gamma(0)=z$ and $\gamma(\tau)=\zeta$), we let $\gamma[\tau,\infty)$ be a chordal \SLEk\ in the configuration $(D_\tau,\zeta,w)$, thereby providing the ``other side'' of the two-sided radial \SLEk\ curve.

We now address the parametrization for these curves, which is to be the natural para\-metriz\-ation for \SLEk.
Lawler and Rezaei proved that if we define the natural time spent by~$\gamma$ in a set $S$ to be the $d$-dimensional Minkowski content
\begin{equation}\label{eq:np}
\Theta(S)=\Cont_d(\gamma\cap S)
=\lim_{\epsilon\to0+}
\epsilon^{d-2}\,\Area\{\zeta:\dist(\zeta,\gamma\cap S)<\epsilon\}
\end{equation}
then the limit exists with probability one, and 
\begin{equation}\label{eq:npmoment}
\E[\Theta(S)]=\int_S G(\zeta)\,dA(\zeta).
\end{equation}

We parametrize the curves by natural parametrization, so that $\Theta(\gamma_t)=t$ for all $t$, and denote by $t_\gamma=\Theta(\gamma)$ the length of a curve $\gamma$.
This allows us to consider the aggregate of the measures $\mu_\zeta$.
To do so, we need to consider an appropriate metric on finite measure spaces.
We shall use the Prokhorov metric induced by the metric on curves
\[
d(\gamma,\gamma')=|t_\gamma-t_{\gamma'}|
+\sup_{0\le s\le1} |\gamma(st_\gamma)-\gamma'(st_{\gamma'})|,
\]
which metrizes the topology of weak convergence under $d$.

Then we may attempt to form the aggregate measure
\begin{equation}\label{eq:agg}
\nu=\int_D \mu_\zeta\,dA(\zeta)
\end{equation}
as a Riemann integral, so long as $\mu_\zeta$ is continuous in $\zeta$. This will be an improper Riemann integral in the sense of Definition~\ref{improper}, where the points $z$ and $w$ need to be excised because the Green's function tends to infinity there. See Section~\ref{sec:measures} for more details on this construction.

The aim of this paper is to relate the aggregate $\nu$ of two-sided radial \SLEk\ to the original chordal \SLEk\ measure $\mu$. 
Consider the analogous discrete situation, in which $\mu$ is a measure on finite-length self-avoiding paths and the pinned measure $\mu_\zeta$ is simply $\mu$ restricted to those paths that pass through $\zeta$.
Then the aggregate $\nu=\sum_\zeta\mu_\zeta$ is simply $\mu$ biased by the number of lattice sites that $\gamma$ traverses.
The continuous analogue of this number is the length of $\gamma$ in the natural parametrization, and we prove the corresponding result:
\begin{thm}
\label{thm:lengthbias}
Let $\kappa\le 4$.
Fix a configuration $D = (D,z,w)$ such that the domain $D$ is bounded and $\p D$ is analytic. Let $\mu$ be chordal \SLEk\ in $D$.
Then the aggregate $\nu$ of two-sided radial \SLEk\ in $D$ as defined in~\eqref{eq:agg} exists as an improper Riemann integral.
It is equal to chordal \SLEk\ biased by the curve's length in the natural parametrization:
\[
\frac{d\nu}{d\mu}(\gamma)=t_\gamma.
\]
\end{thm}

Importantly, we use Zhan's result on reversibility of chordal \SLEk\ in the proof. In particular, it follows from his result that two-sided radial \SLEk\ and length-biased chordal \SLEk\ are both reversible for $\kappa\le4$, and this property will be used to run the curves from the start and the end in alternation, rather than always from the start.

\section{Measures on curves}
\label{sec:measures}

In order to make sense of an integral of measures, we need to consider an appropriate metric space of measures on curves.
Since the length of an SLE curve in the natural parametrization is random, we need a topology on the curves that is robust under small reparametrizations.

\subsection{Parametrized curves}

Let $\K$ denote the set of parametrized curves (continuous paths) $\gamma:[0,t_\gamma]\to\C$, where $0\le t_\gamma<\infty$.
We call $t_\gamma$ the \emph{time duration} of $\gamma$.
The oscillation or modulus of continuity of $\gamma$ is
\[
\osc_\gamma(\delta)=\sup\{|\gamma(s)-\gamma(t)|:|s-t|\le\delta\},
\]
which is continuous in $\delta$.
In particular, $\osc_\gamma(\delta)\equiv0$ if $t_\gamma=0$.

Let $\|\cdot\|$ be the supremum norm.
We can define metrics $\dc$ and $\dd$ on $\K$ by
\begin{align}
\dc(\gamma_1,\gamma_0) &= |t_{\gamma_1}-t_{\gamma_0}|+\|\gamma_1\circ s_1^*-\gamma_0\circ s_0^*\|, \\
\dd(\gamma_1,\gamma_0) &= \inf_{s_1,s_0}(\|s_1-s_0\|
+\|\gamma_1\circ s_1-\gamma_0\circ s_0\|),
\end{align}
where the infimum is taken over all nondecreasing continuous functions $s_i$ on $[0,1]$ sending $0$ to $0$ and $1$ to $t_{\gamma_i}$, the basic example of which is $s_i^*(x)=xt_{\gamma_i}$.
We should think of $\dd$ as detecting curves that are close after a small reparametrization, whereas $\dc$ detects curves that are close after a small \emph{linear} reparametrization.

The idea of allowing small reparametrizations has a long history. In particular, $\dd$ is almost the same as the Skorokhod metric~(see, e.g.,~\cite[Chapter 3]{Billingsley}) except that it allows any nonnegative time duration $t_\gamma$.
The metric $\dd$ coincides with the metric $d_\K$ in~\cite[Chapter 5]{Lbook}, which was defined as an infimum over increasing homeomorphisms. We use the present definition as it clarifies the case $t_\gamma=0$.

It is clear from the definition that $\dd\le\dc$. There is no uniform bound in the other direction; for instance, one can consider $\gamma_n(t)=2^nt\wedge 1$ on $[0,1]$, for which $\dd(\gamma_m,\gamma_n)\le2^{-m}$ and $\dc(\gamma_m,\gamma_n)\ge1/2$ if $m<n$.
This example also shows that $\K$ is not complete under $\dd$.
Completeness is important in order to be able to define an integral of measures as a limit of Riemann sum approximations.
 
Unfortunately, $\K$ is incomplete also under $\dc$, because of the presence of curves of zero length.
This minor inconvenience can be avoided by considering~$\dc$ as a complete metric on $\tilde\K:=[0,\infty)\times C[0,1]$ where the curve $\gamma_0\in\K$ is identified with the pair $(t_{\gamma_0},\gamma_0\circ s_0^*)$.
Since all measures we consider will assign zero measure to zero-length curves, the extra pairs $(0,\tilde\gamma)\in\tilde\K$ with $\tilde\gamma$ non-constant are of no consequence.

\begin{lem}
If $\gamma_1,\gamma_0\in\K$, then
\[
\dd(\gamma_1,\gamma_0)\le\dc(\gamma_1,\gamma_0)\le\dd(\gamma_1,\gamma_0)+\osc_{\gamma_0}(2\dd(\gamma_1,\gamma_0)).
\]
In particular, if $\gamma_0,\gamma_1,\ldots\in\K$ then $\dc(\gamma_n,\gamma_0)\to0$ if and only if $\dd(\gamma_n,\gamma_0)\to0$.
\end{lem}

The lemma implies that $\dd$ is positive definite (so that it is actually a metric) and that $\dc$ and $\dd$ induce the same topology on $\K$.

\begin{proof}
Fix nondecreasing continuous functions $s_i$ on $[0,1]$ sending $0$ to $0$ and $1$ to $t_{\gamma_i}$.
Rescale $s_1$ appropriately for $\gamma_0$ by setting
\[ \tilde s_1(t) = \frac{t_{\gamma_0}}{t_{\gamma_1}}\,s_1(t). \]
We have $|t_{\gamma_1}-t_{\gamma_0}|<\|s_1-s_0\|$ and, since $s_1$ is continuous,
\begin{align*}
\|\gamma_1\circ s_1^*-\gamma_0\circ s_0^*\|
&= \|\gamma_1\circ s_1-\gamma_0\circ\tilde s_1\| \\
&\le \|\gamma_1\circ s_1-\gamma_0\circ s_0\|
 + \|\gamma_0\circ s_0-\gamma_0\circ\tilde s_1\| \\
&\le \|\gamma_1\circ s_1-\gamma_0\circ s_0\|
 + \osc_{\gamma_0}(\|s_0-\tilde s_1\|).
\end{align*}
Moreover,
\[\|s_0-\tilde s_1\|\le \|s_0-s_1\|+\|s_1-\tilde s_1\|=\|s_0-s_1\|+|t_{\gamma_1}-t_{\gamma_0}|\le2\|s_0-s_1\|.
\]
It follows that
\[\dc(\gamma_1,\gamma_0)\le\|s_1-s_0\|+\|\gamma_1\circ s_1-\gamma_0\circ s_0\|+\osc_{\gamma_0}(2\|s_0-s_1\|).\]
The result follows by taking the infimum over all $s_1,s_2$.
\end{proof}

We introduced $\dc$ for its completeness property, but $\dd$ tends to be easier to estimate. One example of this relates to concatenating curves. If $\gamma,\eta\in\K$ with $\gamma(t_\gamma)=\eta(0)$, we define the concatenation $\gamma\oplus\eta$ by
\[
t_{\gamma\oplus\eta}=t_\gamma+t_\eta,\qquad
(\gamma\oplus\eta)(t)=\begin{cases}
\gamma(t),&0\le t<t_\gamma,\\
\eta(t-t_\gamma),&t_\gamma\le t\le t_\gamma+t_\eta.
\end{cases}
\]
Then it is easy to see that
\begin{equation}\label{eq:concat}
\dd(\gamma_1\oplus\eta_1,\gamma_0\oplus\eta_0)\le\dd(\gamma_1,\eta_1)+\dd(\gamma_0,\eta_0),
\end{equation}
but no similar estimate holds for $\dc$.
\subsection{Metric spaces of finite measures}

Let $\B$ denote the Borel sets of $\K$ under the topology induced by $\dd$ (or equivalently $\dc$). 
Let $\M(\K)$ denote the set of finite positive measures on $(\K,\B)$.
If $\mu\in\M(\K)$ we write $|\mu|$ (rather than $\|\mu\|$) for the total mass of $\mu$ in this section, for the sake of legibility.
If $|\mu|\ne0$ we can write $\mu=|\mu|\,\mu^\#$ for a unique probability measure $\mu^\#$.

If $d$ is either $\dd$ or $\dc$, we define the Prokhorov metric on $\M(\K)$ induced by $d$, which we also denote by $d$, as follows:
\[
d(\mu,\nu)=\inf\{\epsilon:\mu(B)\le\nu(B^\epsilon)+\epsilon\text{ and }\nu(B)\le\mu(B^\epsilon)+\epsilon\text{ for all }B\in\B\},
\]
where $B^\epsilon=\{\gamma\in\K:d(\gamma,B)<\epsilon\}$.
This gives two different Prokhorov metrics $\dc$ and $\dd$, since the neighborhood $B^\epsilon$ depends on the choice of metric.
However, they induce the same topology on $\M(\K)$. Indeed, the Prokhorov metric metrizes the topology of weak convergence of probability measures (see~\cite[Appendix III]{Billingsley}), which is the same for $\dc$ as for $\dd$. That they induce the same topology on finite measures then follows from the following lemma.

\begin{lem}
Let $\mu_n\in\M(\K)$. Then $\mu_n\to0$ if and only if $|\mu_n|\to0$. If $\mu\in\M(\K)$ is nonzero then $\mu_n\to\mu$ if and only if $|\mu_n|\to|\mu|$ and $\mu_n^\#\to\mu^\#$.
\end{lem}
\begin{proof}
It is easy to see that
\[
||\mu|-|\nu||\le d(\mu,\nu)\le |\mu|\vee|\nu|,
\]
which implies that $\mu_n\to0$ if and only if $|\mu_n|\to0$.
Moreover, $\mu_n\to\mu$ implies that $|\mu_n|\to|\mu|$.
If $a>1$, we have
\[
d(\mu,\nu)\le d(a\mu,a\nu)\le a\,d(\mu,\nu),\qquad
d(\mu,a\mu)=(a-1)|\mu|.
\]
Hence, the distances $d(\mu_n^\#,\mu^\#)$ and $d(|\mu|\mu_n^\#,\mu)$ differ by a factor of at most $|\mu|\vee|\mu|\inv$.
Moreover, $d(\mu_n,\mu)$ and $d(|\mu|\mu_n^\#,\mu)$ differ by at most 
\[
d(\mu_n,|\mu|\mu_n^\#)=||\mu_n|-|\mu||.
\]
Under the assumption that $|\mu_n|\to|\mu|$, therefore, $\mu_n\to\mu$ if and only if $\mu_n^\#\to\mu^\#$.
\end{proof}

Another useful feature of the Prokhorov metric is additivity: we easily see from the definition that
\begin{equation}\label{prokhadd}
d(\mu+\mu',\nu+\nu')\le d(\mu,\nu) + d(\mu',\nu').
\end{equation}

\subsection{Riemann integrals of measures}

Since $\dc$ makes $\tilde\K$ a complete separable metric space, the Prokhorov metric $\dc$ also makes $\M:=\M(\tilde\K)$ a complete separable metric space. (We must allow the pairs $(0,\tilde\gamma)\in\tilde\K$ here in order for the space to be complete, but all the measures we consider will be supported on curves of positive time duration).

We shall define Riemann integrals as the limits of Riemann sums in the usual way. Definition~\ref{proper} and the results that follow still apply if  any complete separable metric space is used in place of $\tilde\K$.

Some care is needed for this construction because $\M(\tilde\K)$ is not a normed space. 
For this reason, some of the classical approaches to vector-valued integration such as the Bochner integral (see, e.g.,~\cite[Chapter~V]{Yosida}) are not particularly apposite.
The weaker Gelfand-Pettis integral defined using duality (developed, e.g., in~\cite[Chapter 3]{Rudin}) provides a more general theory of vector-valued integration, but we do not use this theory in any of our arguments. One reason for this is that even establishing that the measure-valued function is continuous is a major part of the work of this paper.

\begin{defn}\label{proper}
Let $R\subset\R^d$ be a rectangle (the product of one closed interval in each coordinate axis) and $P$ a tagged partition of $R$, associating to each subrectangle $I\in P$ the sample point $x_I$.
The mesh $\|P\|$ of a partition $P$ is the length of the longest diagonal of any of its subrectangles.
Let $f:R\to\M$ be a measure-valued function.
The Riemann sum of $f$ on $P$ is
\[
\mathcal S_f(P)=\sum_{I\in P} f(x_I) \vol(I).
\]
The (proper) Riemann integral of $f$ on $R$ is the limit
\[
\int_R f(x)\,d\vol(x) = \lim_{\|P\|\to 0} S_f(P)
\]
taken over all partitions $P$ of $R$ as the mesh $\|P\|$ tends to $0$, if the limit exists.

In the case of a measure-valued function $f$ on a general bounded domain $D$ of integration, we define
\[
\int_D f(x)\,d\vol(x) = \lim_{\|P\|\to 0} S_{\hat f}(P)
\]
where $\hat f$ is $f$ extended by $0$ off $D$, provided that the limit exists.
\end{defn}

As with the usual Riemann integral of functions, the proper Riemann integral fails to exist when the total mass of a measure-valued function has limit infinity at a point.
We shall therefore need the notion of an improper integral. For our purposes, it will be simplest to present only the case where one or two points need to be removed from a (possibly unbounded) domain.
\begin{defn}\label{improper}
Let $f$ be a measure-valued function on a (possibly unbounded) domain $D\subset\R^d$ and suppose $|f(x)|$ tends to infinity as $x$ approaches the points $z$ and $w$ in $\ol D$. The improper integral of $f$ on $D$ is defined to be
\[
\int_D f(x)\,d\vol(x):=\lim_{r,s,t\to 0} \int_{D\cap B(0,\frac1r)\sm \ol{B(z,s)}\sm \ol{B(w,t)}} f(x)\,d\vol(x),
\]
if the limit exists.
\end{defn}

In the remainder of this section we present some basic properties of Riemann integrals of continuous measure-valued functions. This material is unnecessary for the proof of Theorem~\ref{thm:lengthbias}, because in that theorem we prove directly the existence of the limit that defines the integral.

\begin{lem}\label{lem:sum1}
If $a_i\ge0$ and $\mu_i,\nu_i\in\M$ for $i=1,\ldots,n$ with $\sum_i a_i=1$ then
\[
\dc\Bigl(\sum_ia_i\mu_i,\sum_ia_i\nu_i\Bigr)\le\sup_i\dc(\mu_i,\nu_i).
\]
\end{lem}

\begin{proof}
Let $V\in\B$ and $\epsilon>\sup_i\dc(\mu_i,\nu_i)$.
As $\mu_i(V)\le\nu_i(V^\epsilon)+\epsilon$ and $\nu_i(V)\le\mu_i(V^\epsilon)+\epsilon$ for all $i$,
we have
\[
\sum_i a_i\mu_i(V)\le\sum_i a_i\nu_i(V^\epsilon)+\epsilon
\quad\text{and}\quad
\sum_i a_i\nu_i(V)\le\sum_i a_i\mu_i(V^\epsilon)+\epsilon.
\qedhere
\]
\end{proof}

Suppose that $f$ is a measure-valued function on a rectangle $R\subseteq \R^d$ that is continuous, hence uniformly continuous, with respect to $\dc$.
For now we shall assume that $R$ is a unit hypercube; this can be undone by scaling later.

\begin{lem}
Suppose $R\subset\R^d$ is a unit hypercube and $P\subset Q$ are nested tagged partitions of $R$; that is, every subrectangle of $Q$ lies in a subrectangle of $P$, but no statement is made about the sample points. Then
\[ \dc(\mathcal S_f(Q),\mathcal S_f(P))\le\osc_f(\|P\|).
\]
\end{lem}

\begin{proof}
We have
\[
\mathcal S_f(P)=\sum_{I\in P}\sum_{\substack{J\in Q\\J\subseteq I}} f(x_I) \vol(J),\qquad
\mathcal S_f(Q)=\sum_{I\in P}\sum_{\substack{J\in Q\\J\subseteq I}} f(x_J) \vol(J).
\]
Write $\epsilon=\osc_f(\|P\|)$.
If $I\in P$ and $J\in Q$ with $J\subseteq I$,
we know that $\dc(f(x_I),f(x_J))\le\epsilon$.
Applying Lemma~\ref{lem:sum1} with weights $\vol(J)$ for $J\in Q$,  we see that $\dc(\mathcal S_f(P),\mathcal S_f(Q))\le\epsilon$.
\end{proof}

By scaling and taking a common refinement if necessary, we obtain:
\begin{cor}
Suppose $R\subset\R^d$ is a rectangle and $P, Q$ are arbitrary tagged partitions of $R$. Then
\[ \dc(\mathcal S_f(Q),\mathcal S_f(P))\le 2\,(\vol(R)\vee1)\,\osc_f(\|P\|).
\]
\end{cor}

\begin{cor}
If $P_n$ are partitions of a rectangle $R$ with $\|P_n\|\to0$ then $(\mathcal S_f(P_n))$ is a Cauchy sequence in $\M$ and has a limit independent of the sequence $(P_n)$. This limit is therefore the integral $\int_R f(x)\,d\vol(x)$ of $f$ on $R$. Moreover,
\[ \dc\Bigl(\mathcal S_f(P),\int_R f(x)\,d\vol(x)\Bigr)\le(\vol(R)\vee1)\,\osc_f(\|P\|).
\]
\end{cor}

It is also easy to see that $\bigl|\int_R f(x)\,d\vol(x)\bigr|=\int_R |f(x)| \,d\vol(x)$.

\section{Generalized escape estimates}

In this section we shall demonstrate that chordal SLE biased by the time spent in a small target region is unlikely to retreat far from that region after first coming close to it.  We derive this result from the uniform estimate for two-sided radial SLE, Theorem~1.3 in~\cite{FL}.

Since we aim to prove that the aggregate of two-sided radial SLE measures is length-biased chordal SLE, it is useful to introduce the following notation for length-biased chordal SLE.

Let $\mu$ be the law of a chordal \SLEk\ curve $\gamma$ in a fixed configuration $(D,z,w)$, which we suppress from the notation.
As in Section~\ref{sec:measures}, we consider this as a Borel measure on the space $\K$ that is carried on naturally parametrized curves.
If $Q\subseteq D$ is a square, or a finite union of squares, we shall denote by $\Theta(Q)$ the time $\gamma$ spends in $Q$ under the natural parametrization, as defined in~\eqref{eq:np}, and let
\[
d\mu_Q = \Theta(Q)\,d\mu,
\]
which is chordal \SLEk\ weighted by the natural time spent in $Q$.
Since $\Theta(\ell)=0$ with probability one for any line $\ell$, as follows from~\eqref{eq:npmoment}, we have
\begin{equation}\label{additive}
\mu_Q=\sum_i\mu_{Q_i}
\end{equation}
if $Q=\bigcup_i Q_i$ and no two squares $Q_i,Q_j$ share an interior point.

As one might expect, the law $\mu_Q^\#$ of chordal \SLEk\ biased by the time spent in a small square $Q$ is close (in the sense of weak convergence) to the law $\mu_z^\#$ of two-sided radial \SLEk\ through any point $z\in Q$.  
This estimate will be made precise in Section~\ref{sec:measuresclose}. 
However, the two measures are mutually singular, and if we want to make any estimates involving Radon--Nikodym derivatives, we need to restrict our attention to the sections of the curve before and after all approaches to $Q$.

The cut-off point that we choose for a square $Q$ is the circle $C(Q)$ with the same center as $Q$ and twice its diameter. (Any fixed ratio of diameters larger than $1$ would have worked for this argument.) 
Then we consider the curve $\gamma$ up to the first time $\rho_Q$ that it hits $C(Q)$, that is,
\[
\rho_Q=\tau_{C(Q)},\quad\text{where}\quad \tau_X = \inf\{t:\gamma(t)\in X\}.
\]

We shall frequently need to consider the reversed curve $\tilde\gamma$ defined by $\tilde\gamma(t)=\gamma(t_\gamma-t)$ and $t_{\tilde\gamma}=t_\gamma$. It is of course true that the map $\gamma\mapsto\tilde\gamma$ is continuous. Reversibility of chordal and two-sided radial \SLEk\ states that if $\gamma$ is an \SLEk\ curve in $D$ from $z$ to $w$ (through $\zeta$, if two-sided) then $\tilde\gamma$ is an \SLEk\ curve in $D$ from $w$ to $z$ (through $\zeta$, if two-sided).

As we will also need to stop the reversed curve $\tilde\gamma$ before it approaches $Q$, we designate the corresponding stopping time by
\[
\tilde\rho_Q=\tilde\tau_{C(Q)},\quad\text{where}\quad
\tilde\tau_X=\inf\{t:\tilde\gamma(t)\in X\}.
\]
Note that this is a stopping time for the reversed curve $\tilde\gamma$ but not for $\gamma$ itself. Nevertheless, the combination of the domain Markov property, reversibility and the uniformity of our escape estimates allows us to deduce the ``whole-curve'' escape estimates that we need.

The first estimate we give is a mild generalization of the two-sided radial escape estimate that appears in \cite[Theorem~1.3]{FL}.

\begin{prop}\label{2srescape}
If $\kappa\le4$, there exists $c<\infty$ such that the following holds.
Let $\gamma$ be a two-sided radial \SLEk\ curve from $z$ to $w$ through $\zeta$ in a domain $D$ that contains $\D$, where $z,w\in\p D\cap\p \D$ and $|\zeta|\le1/2$.
Then for all $r\ge 1$,
\[
\Prob\{\gamma\cap r\,\p\D\ne\emptyset\} \le c\, r^{-(4a-1)/2}.
\]
\end{prop}

\begin{proof}
First consider the initial segment $\gamma[0,\tau_\zeta]$ of the two-sided radial \SLEk\ curve stopped when it hits $\zeta$. Let $S$ be the set of crosscuts of $D$ that are subsets of the circle $r\,\p\D$ and let $\tilde\gamma$ be the line segment from $z$ to $\zeta$. If $\eta\in S$, Corollary~5.6 in~\cite{FL} shows that
\[
\Prob\{\gamma[0,\tau_\zeta]\cap \eta\ne\emptyset\}
\le c\sum_{\eta\in S} \exc_D(\tilde\gamma,\eta)^{4a-1}.
\]
It follows that
\begin{multline*}
\Prob\{\gamma[0,\tau_\zeta]\cap r\,\p\D\ne\emptyset\}
\le c\,\sum_{\eta\in S}\exc_D(\tilde\gamma,\eta)^{4a-1}
\le c\biggl(\sum_{\eta\in S}\exc_D(\tilde\gamma,\eta)\biggr)^{4a-1}\\
\le c\bigl(2\,\exc_D(\tilde\gamma,r\,\p\D)\bigr)^{4a-1}
\le c\,\exc_D(\p\D,r\,\p\D)^{4a-1}
\le c\,e^{-\frac12(4a-1)},
\end{multline*}
where the last inequality follows from the Beurling estimate.

Now we know that the reversal of two-sided radial \SLEk\ in~$D$ from~$z$ to~$w$ through~$\zeta$ is two-sided radial \SLEk\ in~$D$ from~$w$ to~$z$ through~$\zeta$, so the same argument proves that
\[
\Prob\{\tilde\gamma[0,\tilde\tau_\zeta]\cap r\,\p\D\ne\emptyset\}\le c\,r^{-\frac12(4a-1)}.
\]
Since with probability one $\zeta$ is visited exactly once by $\gamma$, so that $\tilde\tau_\zeta=t_\gamma\sm\tau_\zeta$, we may add the two estimates obtained above to yield
\[
\Prob\{\gamma\cap r\,\p\D\ne\emptyset\}\le c\,r^{-\frac12(4a-1)}.\qedhere
\]
\end{proof}

The escape estimate for length-biased chordal \SLEk\ is more complicated because it is not as clear exactly how the bias affects the local behaviour. We start with an estimate on the Green's function for chordal \SLEk, which will allow us to prove that away from the square of interest, length-biased chordal \SLEk\ behaves just like two-sided radial \SLEk.

\begin{prop}[Main Estimate]\label{main-estimate}
If $\kappa\le4$, there exists $c<\infty$ such that the following holds.
Let $(D,z,w)$ be a chordal configuration and $\zeta,\zeta'\in D$ with
\[
r:=\frac{|\zeta-\zeta'|}{\dist(\zeta,\p D)}\le\frac12.
\]
Then
\begin{equation}\label{mainest}
\biggl|\log\frac{G_{D,z,w}(\zeta)}{G_{D,z,w}(\zeta')}\biggr|
\le c\,r.
\end{equation}
\end{prop}

\begin{proof}
The explicit formula~\ref{eq:ghalf} for the Green's function in the upper half-plane implies that
\[
G_{D,z,w}(\zeta') = c_\kappa\, S_{D,z,w}(\zeta')^{4a-1}\,
\crad_D(\zeta')^{d-2},
\]
where $\crad_D(\zeta')$ is the conformal radius of $D$ around $\zeta'$ and $S_{D,z,w}(\zeta')=\sin\arg f(\zeta')$ under any conformal map $f$ sending $D,z,w$ to $\Half,0,\infty$ respectively.
Let $d=\dist(\p D,\zeta)$.
Derivative and Harnack estimates for the positive harmonic function $S_{D,z,w}$ yield
\[
|\nabla S_{D,z,w}(\zeta')|\le c\, \frac{S_{D,z,w}(\zeta')}{d}
\]
when $|\zeta'-\zeta|<\frac12d$ and therefore
\[
\biggl|\log\frac{S_{D,z,w}(\zeta')}{S_{D,z,w}(\zeta)}\biggr|
\le c\,r.
\]

To estimate the conformal radius factor, we use the well-known characterization
\[
\log\crad_D(\zeta')=\E^{\zeta'}\log|B_{\tau}-\zeta'|
\]
for a complex Brownian motion $B$ started at $\zeta'$ and stopped at time $\tau=\inf\{t:B_t\notin D\}$ (see, e.g., \cite[Corollary~3.34]{Lbook}).

Because $|B_{\tau}-\zeta|\ge d$ we have
\[
\biggl|\log\frac{|B_{\tau}-\zeta'|}{|B_{\tau}-\zeta|}\biggr|
\le c\,r.
\]
Furthermore, the distributions of $B_{\tau}$ when the Brownian motion is started from $\zeta$ and from $\zeta'$ are mutually absolutely continuous with Radon--Nikodym derivative $e^{O(r)}$, as can be seen by stopping $B$ first on the circle centered at $\zeta$ of radius $d$.
By applying a similarity transformation, we may assume that $d=1$ and $\zeta=0$. We then have $\log|B_{\tau}|\ge 0$ and therefore
\[
e^{-cr}\,\E^0\log|B_{\tau}|\le \E^{\zeta'}\log|B_{\tau}|
\le e^{cr}\,\E^0\log|B_{\tau}|.
\]
Since Koebe's 1/4-theorem implies that $\E^0\log|B_{\tau}|\le c$ we have that
\[
\bigl|\E^{\zeta'}\log|B_{\tau}|-\E^0\log|B_{\tau}|\bigr|
\le c\,r.
\]
We may conclude that
\[
\biggl|\log\frac{\crad_D(\zeta')}{\crad_D(\zeta)}\biggr|\le c\,r.
\]
The proposition follows from the estimates we have given on the conformal radius and $S_{D,z,w}$ factors.
\end{proof}

An alternative proof of Proposition~\ref{main-estimate} proceeds by estimating the Green's function in $\Half$ after applying conformal covariance and distortion estimates for the map to $\Half$.

For the remaining propositions we fix a chordal configuration $(D,z,w)$ and let $\mu$ denote chordal \SLEk\ in that configuration. If $Q\subset D$ is a square and $\zeta\in Q$, let $\mu_\zeta$ and $\mu_Q$ denote two-sided radial \SLEk\ through $\zeta$ and chordal \SLEk\ weighted by the time spent in $Q$ respectively.

One of the most important consequences for us of this main estimate is the fact that length-biased chordal \SLEk\ is essentially the same as two-sided radial \SLEk, at least before the curve comes too close to the set whose natural parametrization is in question.
To summarize the result, it states that chordal \SLEk\ biased by the time spent in a square $Q$ and stopped (on either side) before coming closer to $Q$ than $R$ times its diameter is mutually absolutely continuous with two-sided radial SLE through any point in $Q$, with Radon-Nikodym derivative uniformly of order $1+O(1/R)$.
The following proposition states the result precisely.

\begin{prop}\label{prop:comparison}
If $\kappa\le4$, there exists $c<\infty$ such that the following holds.
Let $Q$ be a square of diameter $r_0$ and let $C$ be a circle with the same center as $Q$ and radius $r\ge r_0$.
Let $\zeta\in Q$.
Let $(D,z,w)$ be a chordal SLE configuration such that $C\subset \ol D$.
In this configuration, let $\gamma$ be an \SLEk\ curve, considered under the probability measures $\mu_\zeta^\#$ of two-sided radial \SLEk\ through $\zeta$ and $\mu_Q^\#$ of $Q$-length-biased chordal \SLEk.

Then, considered as measures on the pair $\bigl(\gamma[0,\tau_C],\tilde\gamma[0,\tilde\tau_C]\bigr)$, the Radon--Nikodym derivative satisfies
\begin{equation}\label{rnd}
\biggl|\log\frac{d\mu_Q^\#}{d\mu_\zeta^\#}\biggr|
\le c\, \frac{r_0}{r}.
\end{equation}
Moreover, a similar estimate applies to the total masses
\[\|\mu_Q\|=\int_Q\,G_{D,z,w}(\zeta')\,dA(\zeta')\] 
and $\|\mu_\zeta\|=G_{D,z,w}(\zeta)$:
\begin{equation}\label{tmass}
\biggl|\log\frac{\|\mu_Q\|}{\|\mu_\zeta\|\,A(Q)}\biggr|
\le c\, \frac{r_0}{r}.
\end{equation}
\end{prop}

\begin{proof}
The total mass estimate~\eqref{tmass} follows immediately from~\eqref{mainest}.

For the Radon--Nikodym derivative estimate~\eqref{rnd}, it suffices to prove the estimate considering the measures only on the initial segment $\gamma[0,\tau_C]$ and ignoring the final segment $\tilde\gamma[0,\tilde\tau_C]$. The reason for this is as follows. For chordal \SLEk\ conditioned on the initial segment $\gamma[0,\tau_C]$ (and assuming that $\tau_C<\infty$, which holds with probability one under $\mu_Q^\#$ and $\mu_\zeta^\#$), the remainder of the curve $\gamma[\tau_C,t_\gamma]$ follows the law of chordal \SLEk\ in the configuration $(D\sm\gamma[0,\tau_C],\gamma(\tau_C),w)$. The reversal of $\gamma[\tau_C,t_\gamma]$ conditionally follows the law of chordal \SLEk\ in the configuration $(D\sm\gamma[0,\tau_C],w,\gamma(\tau_C))$, by reversibility of chordal \SLEk. Weighting by $\Theta(Q)$ shows that, under $\mu_Q^\#$, the reversal of $\gamma[\tau_C,t_\gamma]$ conditionally follows the law of $Q$-length-biased chordal \SLEk\ in the $(D\sm\gamma[0,\tau_C],w,\gamma(\tau_C))$. A similar argument using reversibility of two-sided radial \SLEk\ shows that, under $\mu_\zeta^\#$, the reversal of $\gamma[\tau_C,t_\gamma]$ conditionally follows the law of two-sided radial \SLEk\ through $\zeta$ in the configuration $(D\sm\gamma[0,\tau_C],w,\gamma(\tau_C))$. We may therefore apply the initial-segment estimate to the measures on the terminal segment $\tilde\gamma[0,\tilde\tau_C]$ \emph{conditional} on the initial segment $\gamma[0,\tau_C]$. Multiplying the Radon--Nikodym derivatives obtained for the measures on the initial segment $\gamma[0,\tau_C]$ and, conditional on that initial segment, on the terminal segment $\tilde\gamma[0,\tilde\tau_C]$, we may conclude that the Radon-Nikodym derivative estimate~\eqref{rnd} applies to the measures on the pair $\bigl(\gamma[0,\tau_C],\tilde\gamma[0,\tilde\tau_C]\bigr)$.

It remains to prove the estimate~\eqref{rnd} considering the measures only on the initial segment $\gamma[0,\tau_C]$.
If we consider only this segment, note that $\mu_Q$ is the same measure as chordal \SLEk\ weighted by the conditional expectation
\[
X:=\E\bigl(\Theta(Q)\,1\{\tau_C<\infty\}\bigm|\gamma[0,\tau_C]\bigr)
= \int_Q G_{D\sm\gamma[0,\tau_C],\gamma(\tau_C),w}(\zeta')\,dA(\zeta').
\]
The integral expression given here follows from~\eqref{eq:np} and the domain Markov property for chordal \SLEk.
Moreover, two-sided radial \SLEk\ $\mu_\zeta$ is the same measure as chordal \SLEk\ weighted by
\[
Y:=G_{D\sm\gamma[0,\tau_C],\gamma(\tau_C),w}(\zeta),
\]
as follows directly from the definition~\eqref{2srdefn}.
By the main estimate~\eqref{mainest},
\[
\biggl|\log\frac{X}{Y\,A(Q)}\biggr|\le c\,\frac{r_0}{r}.
\]
Dividing by the total masses to obtain the probability measures $\mu_Q^\#$ and $\mu_\zeta^\#$, and using~\eqref{tmass}, we obtain~\eqref{rnd} for the initial segment $\gamma[0,\tau_C]$.
\end{proof}

An easy consequence is the escape estimate for $Q$-length-biased chordal \SLEk\ \emph{away from}~$Q$.

\begin{cor}\label{newcor}
If $\kappa\le4$, there exists $c<\infty$ such that the following holds.
Let $\gamma$ be a chordal \SLEk\ curve from $z$ to $w$ in a domain $D$ that contains $\D$, where $z,w\in\p D\cap\p \D$, \emph{weighted} by the natural time spent in a square $Q\subset\frac12\D$. Then for all $r\ge 1$,
\begin{equation}\label{lbcescape}
\Prob\bigl\{(\gamma[0,\rho_Q]\cup\tilde\gamma[0,\tilde\rho_Q])\cap r\,\p\D\ne\emptyset\bigr\} \le c\, r^{-(4a-1)/2}.
\end{equation}
\end{cor}

\begin{proof}
Let $\zeta\in Q$, and apply~\eqref{rnd} to $\mu_Q^\#$ and $\mu_\zeta^\#$ with the circle $C=C(Q)$ (so that $r=r_0$). We see that the Radon--Nikodym derivative between the two probability measures is uniformly bounded away from $0$ and $\infty$. Since the escape estimate was established for $\mu_\zeta^\#$ in Proposition~\ref{2srescape}, it carries over on the initial and final segments to yield~\eqref{lbcescape}.
\end{proof}

More work is needed to establish the escape estimate for the \emph{entire} length-biased chordal \SLEk\ curve, which we state in the following theorem.

\begin{thm}\label{alltime}
If $\kappa\le4$, there exists $c<\infty$ such that the following holds.
Let $\gamma$ be a chordal \SLEk\ curve from $z$ to $w$ in a domain $D$ that contains $\D$, where $z,w\in\p D\cap\p \D$, \emph{weighted} by the natural time spent in a square $Q\subset\frac12\D$. Then for all $r\ge 1$,
\[
\Prob\bigl\{\gamma\cap r\,\p\D\ne\emptyset\bigr\} \le c\, r^{-(4a-1)/2}.
\]
\end{thm}

The idea of the proof of Theorem~\ref{alltime} is to subdivide $Q$ into $2^{2k}$ equal subsquares, to apply Corollary~\ref{newcor} to each of those subsquares, and to average the resulting estimates.
It seems, however, that after dividing $Q$ into smaller squares, we still have the problem of understanding the intermediate segments of $\gamma$ that actually pass through those squares.

The observation that solves this problem is that after subdividing $Q$ a finite number of times \emph{that depends on $\gamma$}, there can be no escape during those intermediate segments.
To make the argument precise, we need the following lemma, which states that it suffices to subdivide $Q$ a finite (but $\gamma$-dependent, i.e., random) number of times.

\begin{lem}\label{compactness}
Let $Q$ be a square and
 $\gamma$ a simple curve starting and ending outside $C(Q)$.
Suppose that $\gamma$ intersects a closed set $\eta$ that lies outside $C(Q)$.
There is a positive integer $k$ such that, if $\mathcal Q_k$ is the set of $2^{2k}$ equal subsquares of $Q$ obtained by subdividing $Q$ dyadically $k$ times, then for each $Q'\in\mathcal Q_k$,
either the initial segment $\gamma[0,\rho_{Q'}]$ or the terminal segment $\tilde\gamma[0,\tilde\rho_{Q'}]$ intersects~$\eta$.
\end{lem}

\begin{proof}
Let $D$ be the disk bounded by $C(Q)$.
Since the set $S:=\{t:\gamma(t)\in \ol D\}$ is compact, we may find finitely many disjoint time intervals $I_1,\ldots,I_n$ 
that cover $S$ and such that 
the images $\gamma(I_1),\ldots,\gamma(I_n)$ do not intersect $\eta$.
Since $\gamma$ is simple, there is a minimum distance $\delta>0$ 
between any pair $\gamma(I_j),\gamma(I_k)$ of these images.
Now consider a dyadic subdivision $\mathcal Q_k$ of $Q$ 
into squares of side length less than $\delta/(2\sqrt2)$.
Fix any square $Q'\in\mathcal Q_k$ and observe that $\diam C(Q')<\delta$.
It follows that the intermediate segment~$\gamma[\rho_{Q'},t_\gamma-\tilde\rho_{Q'}]$ 
cannot intersect $\eta$.
For if the intermediate segment intersected $\eta$, there would be distinct time intervals $I_j$ and $I_k$ such that $\gamma(I_j)$ and $\gamma(I_k)$ both contained points inside $C(Q')$, implying that $\dist(\gamma(I_j),\gamma(I_k))<\delta$ which is impossible by our construction.
Since $\gamma$ intersects $\eta$, 
the initial segment $\gamma[0,\rho_{Q'}]$ or final segment $\tilde\gamma[0,\tilde\rho_{Q'}]$ (or both) must therefore intersect $\eta$.
\end{proof}

\begin{proof}[Proof of Theorem~\ref{alltime}]
Let $E$ denote the event that $\gamma\cap r\,\p\D\ne\emptyset$.
Introduce the dyadic subdivisions $\mathcal Q_k$ of $Q$, as in the proof of Lemma~\ref{compactness}.
Let $E_k$ denote the event that for each $Q'\in\mathcal Q_k$, either the initial segment $\gamma[0,\rho_{Q'}]$ or the terminal segment $\tilde\gamma[0,\tilde\rho_{Q'}]$ (or both) intersects~$\eta$.
This is an increasing sequence of events and
by Lemma~\ref{compactness}, $\bigcup_k E_k=E$.

We have the estimate~\eqref{lbcescape} for each subsquare $Q'\in\mathcal Q_k$.
The additivity of $\mu_Q$ across disjoint subsquares~\eqref{additive} implies that the probability of the event $E_k$ under $\mu_Q^\#$ is a weighted average of the probabilities under $\mu_{Q'}^\#$ where $Q'$ runs through the subsquares:
\[
\mu_Q^\#(E_k)=\sum_{Q'\in\mathcal Q_k} 
\frac{\|\mu_{Q'}\|}{\|\mu_Q\|} \,\mu_{Q'}^\#(E_k)
\]
By~\eqref{lbcescape}, we have 
\[\mu_{Q'}^\#(E_k)\le c\, r^{-\frac12(4a-1)}\] 
for each $Q'\in\mathcal Q_k$,
with a constant $c$ that depends only on $\kappa$.
Taking a weighted average, we also have
\[\mu_Q^\#(E_k)\le c\,r^{-\frac12(4a-1)}.\] 
Since $\mu_Q^\#(E_k)\uparrow\mu_Q^\#(E)$, the escape estimate for the entire curve $\gamma$ under $Q$-length-biased chordal \SLEk\ follows.
\end{proof}

We now know that chordal SLE biased by the time spent in a small square $Q$ is unlikely to escape far from a circle $C$ around $Q$ between the first and last times that it visits $C$.

We remark that the statements of Proposition~\ref{prop:comparison} and Theorem~\ref{alltime} remain valid if $Q$ is a rectangle rather than a square, with the same proofs.

The next section will show that the time spent while near the square is also small, with high probability.

\section{Time spent near a small square}

The main tools for bounding the time spent near a small
square are the up-to-constant bounds on the two-point
Green's function established by Lawler and Rezaei
in~\cite{basicpropnew}.
Because the expected natural time spent by SLE in a region
is simply the integral of the Green's function over that
region, these bounds provide a simple first-moment bound on
the time spent that suffices for the proof of
Theorem~\ref{thm:lengthbias}.

\begin{thm}\label{11sep}
If $\kappa<8$, there exists $c<\infty$ such that the
following holds.
Let $D=(D,z,w)$ be a configuration and $U\subset D$ a disk such that
$\dist(U,\p D)\ge 3\diam(U)$.
Let $\gamma$ be either two-sided radial \SLEk\ through a
point $\zeta\in U$ or chordal \SLEk\ weighted by the time spent
in a subset $S$ of $U$.
Then $\E[\Theta(U)]\asymp\diam(U)^d$.
In particular, $\Prob\{\Theta(U)>\epsilon\}<\epsilon$ as
long as $\diam(U)<c_1\epsilon^{2/d}$, where $c_1$ is a
constant depending only on $\kappa$.
\end{thm}

The theorem follows by a simple first moment estimate,
which rests on the following lemma.
Denote by $G(\zeta,\omega)$ the (Euclidean) unordered two-point Green's function for chordal \SLEk\ as considered in~\cite{minkowski},
\[
G(\zeta,\omega)=\lim_{\epsilon,\delta\to0}(\epsilon\delta)^{d-2}\,
\Prob\{\dist(\gamma,\zeta)<\epsilon,\dist(\gamma,\omega)<\delta\}.
\]

\begin{lem}\label{lem11sep}
Under the hypotheses of Theorem~\ref{11sep},
\[
\E[\Theta(U)] =
\frac{\int_U \int_U G(\zeta,\omega) \, dA(\omega) \, d\lambda(\zeta)}
{\int_U G(\zeta) \, d\lambda(\zeta)},
\]
where $\lambda$ is a point mass at $\zeta$ in the case of
two-sided radial SLE through $\zeta$, or the area measure on
the subset $S$ in the weighted chordal case.
\end{lem}

\begin{proof}
From Zhan's proof in~\cite{ZhanC} that chordal \SLEk\ is reversible for $\kappa\le4$, it follows that two-sided radial \SLEk\ is also reversible; that is, the reversal of two-sided radial \SLEk\ in $D$ from $z$ to $w$ through $\zeta$ has the law of two-sided radial \SLEk\ in $D$ from $w$ to $z$ through $\zeta$.
To use this fact, we decompose $\Theta(U)$ (for fixed $U$) as the sum of $\Theta_1$ and $\Theta_2$, these being the natural time spent by the two-sided radial \SLEk\ curve $\gamma$ in $U$ before and after visiting $\zeta$, respectively.
The law of $\gamma[\tau_\zeta,t_\gamma]$, conditional on the initial segment of $\gamma$ up to the time $\tau_\zeta$ that it hits $\zeta$, is simply chordal \SLEk\ in the slit domain from $\zeta$ to $w$.
Hence the conditional expected time spent in $U$ is
\[
\E[\Theta_2\mid\F_{\tau_\zeta}] = \int_U G_{D_{\tau_\zeta},\zeta,w}(\omega)\,dA(\omega).
\]
Taking an unconditional expectation,
\[
\E[\Theta_2] = \int_U \E[G_{D_{T_\zeta},\zeta,w}(\omega)]\,dA(\omega)
= \int_U \frac{\hat G(\zeta,\omega)}{G(\zeta)} \,dA(\omega)
\]
where $\hat G(\zeta,\omega)$ is the (Euclidean) ordered two-point Green's function for chordal \SLEk~\cite{multipoint,minkowski} in the configuration $(D,z,w)$.
Applying reversibility of two-sided radial SLE, we may similarly write
\[
\E[\Theta_1] = \int_U \frac{\hat G(\omega,\zeta)}{G(\zeta)}\,dA(\omega).
\]
Since the unordered Green's function is the sum of the two ordered Green's functions for that pair of points~\cite{multipoint}, the lemma follows in the two-sided radial case.

In the case of chordal \SLEk\ weighted by the time spent in a subset $S$ of $U$, which we denote $\mu_S$, we have
\[
\E_{\mu_S}[\Theta(U)] = \frac{\E[\Theta(U)\Theta(S)]}{\int_S G(\zeta)\,dA(\zeta)}
= \frac{\int_S \int_U G(\zeta,\omega)\,dA(\omega)\,dA(\zeta)}{\int_S G(\zeta)\,dA(\zeta)}.
\]
The last equality follows from the relation
\[
\E[\Theta(A)^2]=\int_A\int_A G(\zeta,\omega)\,dA(\omega)\,dA(\zeta),
\]
proved in~\cite{minkowski}, and linearity.
\end{proof}

\begin{proof}[Proof of Theorem~\ref{11sep}]
We first recall the estimate on the unordered two-point Green's function that appeared in~\cite{basicpropnew}.
Namely, for chordal \SLEk\ in $\Half$ from $0$ to $\infty$, if $\beta=4a-1+d-2=\frac\kappa8+\frac8\kappa-2$ and $|\zeta|<|\omega|$ with $q=|\zeta-\omega|/|\omega|$, then
\[ G(\zeta,\omega)\asymp G(\zeta)\, G(\omega)\, q^{d-2}\, (S(\zeta)\vee q)^{-\beta}. \]
Substituting the known form of $G(\omega)$, we see that
\[ G(\zeta,\omega)\asymp G(\zeta)\, |\zeta-\omega|^{d-2}\, \Bigl(\frac q{S(\omega)}\vee 1\Bigr)^{-\beta}. \]
From the assumption that $\dist(U,\p D)\ge 3\diam(U)$, we can use distortion estimates to see that the last factor is of constant order in our case.

Return to considering \SLEk\ in $D$ from $p$ to $q$.
It is now simple to see that, for any $z\in U$,
\[
\int_U G(\zeta,\omega)\,dA(\omega)
\asymp G(\zeta)\int_U |\zeta-\omega|^{d-2}\,dA(\omega)
\asymp \diam(U)^d.
\]
By Lemma~\ref{lem11sep} and the fact that $G(\zeta)\asymp G(\omega)$ for $\zeta,\omega\in U$, we see that under two-sided radial or weighted chordal \SLEk\ as per the statement,
\[
\E[\Theta(U)]\asymp \diam(U)^d. \qedhere
\]
\end{proof}

\section{Estimates on the distance between length-biased chordal \SLEk\ and the aggregate of two-sided radial \SLEk}
\label{sec:measuresclose}

The integral of measures is defined as the limit of its Riemann sum approximations as the mesh of the partition tends to zero (see Section~\ref{sec:measures}). 
We shall prove directly that length-biased chordal \SLEk\ is the limit of these Riemann sums.

To make this argument, we introduce a metric $d_S$ on finite measures $\mu,\nu$ on $\K$ (the space of parametrized curves introduced in Section~\ref{sec:measures}), given by
\[
d_S(\mu,\nu)=\biggl|\log\frac{\|\mu\|}{\|\nu\|}\biggr|
+\dd(\mu^\#,\nu^\#).
\]
The advantage of this metric over the Prokhorov metric $\dd$ introduced earlier is that it captures the notion of relative error and is stable (up to a constant) under addition.
Indeed, we have the following lemma.
\begin{lem}\label{ds-additive}
Suppose $I$ is a finite set and, for $i\in I$, $\mu_i$ and $\nu_i$ are measures in $\K$. If $d_S(\mu_i,\nu_i)\le\epsilon\le1$ for all $i\in I$, then
\[
d_S\biggl(\sum_i\mu_i,\sum_i\nu_i\biggr)
\le 18\,\epsilon.
\]
\end{lem}
\begin{proof}
Write $\mu=\sum_i\mu_i$ and $\nu=\sum_i\nu_i$.
It is clear that
\[\biggl|\log\frac{\|\mu\|}{\|\nu\|}\biggr|\le\epsilon.\]
For any $i\in I$ and Borel set $A\subset \K$, we have $\mu_i^\#(A)\le\nu_i^\#(A^\epsilon)+\epsilon$ and therefore
\[\mu_i(A)\le\frac{\|\mu_i\|}{\|\nu_i\|}\nu_i(A^\epsilon)+\|\mu_i\|\epsilon.\]
Using that $\|\mu_i\|\le e^\epsilon\,\|\nu_i\|$ and summing over $i$, we have the bound
\[\mu(A)\le e^\epsilon\,\nu(A^\epsilon)+\|\mu\|\,\epsilon
\le \nu(A^\epsilon) +c\,\|\mu\|\,\epsilon\]
which translates readily into the following bound for the probability measure:
\[
\mu^\#(A) 
\le \nu^\#(A^\epsilon)\, e^\epsilon\, \frac{\|\mu\|}{\|\nu\|} + \epsilon
\le \nu^\#(A^\epsilon)\, e^{2\epsilon} + \epsilon
\le \nu^\#(A^\epsilon) + 17\,\epsilon,
\]
where the last inequality uses that $e^{2\epsilon}-1 \le 2\epsilon\,(e^2-1)< 16\epsilon$.
Swapping $\mu$ and $\nu$, we conclude that $\dd(\mu^\#,\nu^\#)\le 17\,\epsilon$.
\end{proof}

The key idea to the proof is that, as suggested by Proposition~\ref{prop:comparison}, there is not much difference between chordal \SLEk\ weighted by the time spent in a square $Q$ and two-sided radial SLE through a point $\zeta\in Q$, at least for the initial and final segments away from that square.
The remaining segment connecting the initial and terminal segments is short (in space and time) with high probability.
We assemble these results in the following proposition.

\begin{prop}\label{smallscale}
If $\kappa\le4$, there exist $\alpha>0$, $c<\infty$ such that the following holds.
Let $Q$ be a square of side length $\epsilon$ at distance at least $\delta\in(0,1]$ from the boundary of a domain $D$, and let $\mu$ be chordal \SLEk\ in $D$. Let $\zeta,\zeta'$ be points in $Q$. Assume that $\epsilon<\delta/100$. Then
\[
d_S(\mu_Q,\, \mu_\zeta\, A(Q)) \le c\,(\epsilon/\delta)^\alpha.
\]
In particular, $d_S(\mu_{\zeta'}, \mu_\zeta) \le c\,(\epsilon/\delta)^\alpha$.
\end{prop}

\begin{proof}
By scaling the given configuration up, it suffices to assume that $\delta=1$. (Observe that distance and time each increase by a fixed factor under scaling up and hence the $\dd$ distance between curves can only increase).

By~\eqref{tmass} we have 
\[
\biggl|\log\frac{\|\mu_Q\|}{\|\mu_\zeta\|\,A(Q)}\biggr|
\le c\,\epsilon.
\]
Now let $C$ be the circle around $Q$ of radius $\epsilon^{1/2}$ and recall the corresponding stopping times $\tau_C$ for $\gamma$ and $\tilde\tau_C$ for $\tilde\gamma$. Considering the probability measures as measures solely on the pair $(\gamma_\tau,\tilde\gamma_{\tilde\tau})$,
by~\eqref{rnd} we have
\[
\Biggl|\log\frac{d\mu_Q^\#}{d\mu_\zeta^\#}\Biggr|
\le c\, \epsilon^{1/2}.
\]

We must address the behaviour of the intervening segment $\gamma[\tau_C,t_\gamma-\tilde\tau_C]$: conditional on $(\gamma_\tau,\tilde\gamma_{\tilde\tau})$, the intervening segment of $\gamma$ is unlikely to go far in space or time.

By Theorem~\ref{alltime}, the probability under $\mu_Q^\#$ that $\gamma[\tau_C,t_\gamma-\tilde\tau_C]$ leaves the circle $C'$ of radius $\epsilon^{1/4}$ around $Q$ is at most $c\,\epsilon^{(4a-1)/8}$.
By \cite[Theorem~1.3]{FL}, the same estimate holds under two-sided radial \SLEk, $\mu_\zeta^\#$.

Finally, by Theorem~\ref{11sep}, the probability (under either $\mu_Q^\#$ or  $\mu_\zeta^\#$) that more than $c\,\epsilon^{d/8}$ time is spent inside $C'$ is at most $c\,\epsilon^{d/8}$, with a constant $c$ depending only on $\kappa$. (Note that the square $Q$ here is actually the set $S$ in the application of Theorem~\ref{11sep}.)

Choose $\alpha=\min(4a-1,d)/8<1/4$. Denote $\gamma_2:=\gamma[\tau_C,t_\gamma-\tilde\tau_C]$ and let $\hat\gamma_2$ be the segment $[\gamma(\tau_C),\tilde\gamma(\tilde\tau_C)]$ (traversed over the same time interval of length $t_\gamma-\tilde\tau_C-\tau_C$). We have, therefore, under both $\Prob=\mu_Q^\#$ and $\Prob=\mu_\zeta^\#$,
\[
\Prob\bigl\{\dd(\gamma_2,\hat\gamma_2)>c\,\epsilon^\alpha\mid\gamma_{\tau_C},\tilde\gamma_{\tilde\tau_C}\bigr\} < c\,\epsilon^\alpha.
\]
Since $\gamma$ is the concatenation of $\gamma_1:=\gamma_{\tau_C}$, $\gamma_2$ and the reversal $\gamma_3$ of $\tilde\gamma_{\tilde\tau_C}$, it follows from~\eqref{eq:concat} that 
\[
\Prob\bigl\{\dd(\gamma,\gamma_1\oplus\hat\gamma_2\oplus\gamma_3)
> c\,\epsilon^\alpha\mid\gamma_{\tau_C},\tilde\gamma_{\tilde\tau_C}\bigr\} 
< c\,\epsilon^\alpha.
\]
Finally, the Radon-Nikodym derivative estimates above show that we may couple the measures $\mu_Q^\#$ and $\mu_\zeta^\#$ on initial and final segments $(\gamma_1,\gamma_3)$ and $(\eta_1,\eta_3)$ under a joint law $\Prob$ such that
\[
\Prob\bigl\{\gamma_1\oplus\hat\gamma_2\oplus\gamma_3 \ne
\eta_1\oplus\hat\eta_2\oplus\eta_3\bigr\}
< c\,\epsilon^\alpha.
\]
In this coupling $\Prob$, conditional on $\gamma_1,\gamma_3,\eta_1,\eta_3$, let us declare the intermediate segments $\gamma_2$ and $\eta_2$ to be independent and follow the laws of $Q$-length-biased chordal \SLEk\ and two-sided radial \SLEk\ through $\zeta$ in the doubly slit domains $D\sm\{\gamma_1,\gamma_3\}$ and $D\sm\eta_1,\eta_3$ respectively.
Then the completed curves $\gamma=\gamma_1\oplus\gamma_2\oplus\gamma_3$ and $\eta=\eta_1\oplus\eta_2\oplus\eta_3$ are marginally distributed as $Q$-length-biased chordal \SLEk\ and two-sided radial \SLEk\ through $\zeta$ in $D$, respectively.
Adding the above estimates yields the estimate
\begin{align*}
\Prob\bigl\{\dd(\gamma,\eta)>2c\,\epsilon^{\alpha}\bigr\}
&\le \Prob\bigl\{\dd(\gamma,\gamma_1\oplus\hat\gamma_2\oplus\gamma_3)>c\,\epsilon^\alpha\bigr\}\\
&\quad + \Prob\bigl\{\dd(\eta,\eta_1\oplus\hat\eta_2\oplus\eta_3)>c\,\epsilon^{\alpha}\bigr\}\\
&\quad + \Prob\bigl\{\gamma_1\oplus\hat\gamma_2\oplus\gamma_3\ne
\eta_1\oplus\hat\eta_2\oplus\eta_3\bigr\} \le 3c\,\epsilon^{\alpha},
\end{align*}
and hence $\dd(\mu_Q^\#,\mu_z^\#)<3c\,\epsilon^{\alpha}$, as required.
\end{proof}

\begin{proof}[Proof of Theorem~\ref{thm:lengthbias}]

We first establish that the proper Riemann integrals that appear in Definition~\ref{improper} over the truncated domain $D_{r,s,t}=D\cap B(0,\frac1r)\sm \ol{B(z,s)}\sm \ol{B(w,t)}$ are $D_{r,s,t}$-length-biased chordal \SLEk:
\begin{equation}\label{drst}
\int_{D_{r,s,t}} f(\zeta)\,dA(\zeta) = \mu_{D_{r,s,t}}.
\end{equation}
The use of $B(0,\frac1r)$ is unnecessary for the theorem as stated, but see Remark~\ref{endrmk}.

Let $\epsilon\in(0,1)$ be arbitrary.
Consider a rectangle $R$ containing $D_{r,s,t}$ and a tagged partition~$\mathcal P$ of $R$ with mesh $\|\mathcal P\|<\epsilon$ and arbitrary  sample points $z_Q\in Q$ for each $Q\in\mathcal P$.

We shall divide the subrectangles of $\mathcal P$ into two groups. Let $I$ (the ``interior'' rectangles) consist of those subrectangles $Q\in\mathcal P$ that lie at distance $\epsilon^{1/2}$ or greater from $\C\sm D_{r,s,t}$, and let $J = \mathcal P \sm I$.

For each $Q\in I$,  we have $d_S(\mu_Q,\,\mu_{z_Q}\,A(Q))\le c\,\epsilon^{\alpha/2}$ by Proposition~\ref{smallscale}.
It follows by Lemma~\ref{ds-additive} that 
\[
d_S\Bigl(\mu_I, \sum_{Q\in I} \mu_{z_Q}\,A(Q)\Bigr) \le c\,\epsilon^{\alpha/2},
\]
where $\mu_I=\sum_{Q\in I} \mu_Q$ by~\eqref{additive}.
(By a convenient abuse of notation, we write $\mu_I$ for $\mu_{\bigcup_{Q\in I} Q}$.)

We also know the total mass of $\mu_I$ is $\int_I G_D(\zeta)\,dA(\zeta)$, which is of order $1$ for $\epsilon$ small enough. (Recall that the configuration $D$ is fixed throughout this proof.) This is because $G_D$ is integrable on $D_{r,s,t}$. 

We may therefore convert the $d_S$ estimate into the following estimate on the Prokhorov distance between the measures in question:
\[
\dd\Bigl(\mu_I, \sum_{Q\in I} \mu_{z_Q}\,A(Q)\Bigr) \le c_D\,\epsilon^{\alpha/2},
\]
where $c_D$ is a constant that depends only on $D$ (and $\kappa$, since our generic constants $c$ may depend on $\kappa$).

Let us now consider the remaining subrectangles $Q\in J$. The key observation here is that, even if we do not control the curve's behaviour under $\mu_Q$ and $\mu_{z_Q}$, the total mass of these measures is insufficient to make a significant contribution to the total. Indeed, since $D_{r,s,t}$ has bounded, piecewise analytic boundary, the total area of the rectangles in $J$ that actually intersect $\ol D_{r,s,t}$ is at most $c_{D_{r,s,t}}\, \epsilon^{1/2}$, where $c_{D_{r,s,t}}$ depends only on $D_{r,s,t}$. The total masses
\[
\|\mu_J\| = \int_{D_{r,s,t}\cap \bigcup_{Q\in J} Q} G_D(\zeta)\,dA(\zeta)
\]
and
\[
\Bigl\|\sum_{Q\in J} \mu_{z_Q}\,1_{z_Q\in D_{r,s,t}}\,A(Q)\Bigr\|
= \sum_{Q\in J:\,z_Q\in D_{r,s,t}} G_D({z_Q})\,A(Q)
\]
are both at most $c_{D_{r,s,t}}\, \epsilon^{1/2}$.

Using the additivity of the Prokhorov metric $\dd$, as seen in~\eqref{prokhadd}, we may conclude that
\[
\dd\Bigl(\mu_I+\mu_J,\sum_{Q\in\mathcal P} \mu_{z_Q}\,\,1_{z_Q\in D_{r,s,t}}\,A(Q)\Bigr) \le c_{D_{r,s,t}} \,\epsilon^{\alpha/2},
\]
and since $\mu_{D_{r,s,t}} = \mu_I + \mu_J$, we have proved~\eqref{drst}.
\medskip

Now we need to take an increasing limit of the improper integrals in~\eqref{drst}. We will need to use one more fact about the Green's function, namely that 
\begin{equation}\label{gd-integrable}
\int_D G_D(\zeta)\,dA(\zeta)<\infty.
\end{equation}
Observe that $G_D$ is continuous on $\ol D$ except at $z$ and $w$, and is locally integrable at $z$ and~$w$. This follows, after applying a conformal map, from the integrability of $G_\Half$ near $0$ and the fact that $G_\Half$ is invariant under the inversion $z\mapsto-1/z$.
Second, $G_D$ is continuous on $\ol D$ except at $z$ and $w$, with value $0$ on $\p D\sm\{z,w\}$.
As $D$ is bounded, we have~\eqref{gd-integrable}.

We may now simply note that
\[
\mu_D = \mu_{D_{r,s,t}} + \mu_{D\sm D_{r,s,t}}
= \int_{D_{r,s,t}} f(\zeta)\,dA(\zeta) + \mu_{D\sm D_{r,s,t}}
\]
and therefore, by~\eqref{gd-integrable},
\[
\dd\Bigl(\mu_D,\int_{D_{r,s,t}} f(\zeta)\,dA(\zeta)\Bigr)
\le\|\mu_{D\sm D_{r,s,t}}\| = \int_{D\sm D_{r,s,t}} G_D(\zeta)\,dA(\zeta) \to 0
\]
as $r,s,t\to 0$. Therefore length-biased chordal \SLEk\ $\mu_D$ is equal to the improper integral
\[
\mu_D = \int_D G_D(\zeta)\,dA(\zeta).
\qedhere
\]
\end{proof}

We conclude with some remarks on the proof.

\begin{rmk}\label{endrmk}
The case where $D$ is unbounded can be handled by the same improper integral mechanism, but we do need to verify~\eqref{gd-integrable}.
Suppose still that $D$ has analytic boundary, including at $\infty$, meaning that $\p(1/D)$ is analytic at $0$.
The question is whether $G_D$ is integrable at $\infty$. This is certainly false if $z$ or $w$ is $\infty$. Otherwise, the following computation (in the case $(D,z,w)=(\Half,x,y)$ for simplicity, though the result is the same for any $D$) can be used to verify that this will be the case if and only if $4a>\sqrt2+1$, that is, if and only if $\kappa < 8(\sqrt2-1)=3.3137...$:
\begin{align*}
\int_{\zeta\in\Half:|\zeta|>R^{-1}} G_{\Half,x,y}(\zeta)\,dA(\zeta)
&=\int_{z\in\Half:|z|<R} G_{\Half,-1/x,-1/y}(z)\,|z|^{2(2-d)}\,\frac{dA(z)}{|z|^4}\\
&\asymp c_{x,y,\kappa} \int_{z\in\Half:|z|<R} (\Im z)^{4a+d-3}\,|z|^{-2d}\,dA(z)\\
&\asymp c_{x,y,\kappa} \int_0^R r^{4a-\frac1{4a}-3}\,dr.
\end{align*}
\end{rmk}

\begin{rmk}
If we allow $D$ to have a non-analytic boundary, the result is still true so long as we can prove the ``discarded'' parts of the measure are still small. In particular, if $G_D$ is integrable and the area of the $\epsilon^{1/2}$-expansion $(\p D)^{\epsilon^{1/2}}$ decays as a power of $\epsilon$ as $\epsilon\to 0$, we can still apply the same proof with a different power of $\epsilon$. This certainly encompasses such cases as bounded domains with piecewise $C^1$ boundaries.
\end{rmk}

\end{document}